\documentclass[letterpaper,11pt,twoside,keywordsasfootnote,addressatend,noinfoline]{article}
\usepackage{fullpage}
\usepackage[english]{babel}
\usepackage{amssymb}
\usepackage{amsmath}
\usepackage{theorem}
\usepackage{epsfig}
\usepackage{subfigure}
\usepackage{mathrsfs}
\usepackage{color}
\usepackage{imsart}

\theorembodyfont{\normalfont}
\newtheorem{theorem}{Theorem}
\newtheorem{lemma}[theorem]{Lemma}

\newenvironment{proof}{\noindent{\bf Proof.}}{\hspace*{2mm}~$\square$}
\newenvironment{proofof}[1]{\noindent{\bf Proof of #1.}}{\hspace*{2mm}~$\square$}

\newcommand{\N}{\mathbb{N}}

\newcommand{\R}{\mathbb{R}}
\newcommand{\A}{\mathscr A}
\newcommand{\B}{\mathscr B}
\newcommand{\ind}{\mathbf{1}}

\newcommand{\n}{\hspace*{-5pt}}

\DeclareMathOperator{\card}{card}
\DeclareMathOperator{\uniform}{Uniform \,}

\begin{document}

\begin{frontmatter}
\title     {Rigorous proof of the Boltzmann-Gibbs distribution \\ of money on connected graphs}
\runtitle  {Boltzmann-Gibbs distribution of money}
\author    {Nicolas Lanchier\thanks{Research supported in part by NSA Grant MPS-14-040958.}}
\runauthor {N. Lanchier}
\address   {School of Mathematical and Statistical Sciences \\ Arizona State University \\ Tempe, AZ 85287, USA.}

\maketitle

\begin{abstract} \ \
 Models in econophysics, i.e., the emerging field of statistical physics that applies the main concepts of traditional physics to economics,
 typically consist of large systems of economic agents who are characterized by the amount of money they have.
 In the simplest model, at each time step, one agent gives one dollar to another agent, with both agents being chosen independently and
 uniformly at random from the system.
 Numerical simulations of this model suggest that, at least when the number of agents and the average amount of money per agent are large,
 the distribution of money converges to an exponential distribution reminiscent of the Boltzmann-Gibbs distribution of energy in physics.
 The main objective of this paper is to give a rigorous proof of this result and show that the convergence to the exponential distribution
 is universal in the sense that it holds more generally when the economic agents are located on the vertices of a connected graph and interact
 locally with their neighbors rather than globally with all the other agents.
 We also study a closely related model where, at each time step, agents buy with a probability proportional to the amount of money they
 have, and prove that in this case the limiting distribution of money is Poissonian.
\end{abstract}



\end{frontmatter}


\section{Introduction}
\label{sec:intro}

\indent This paper is concerned with variants of one of the simplest models in the relatively new field of
 econophysics~\cite{mantegna_stanley_1999, stanley_et_al_1996}, the branch of statistical physics focusing on problems in economics and finance, and also
 a subfield of sociophysics~\cite{galam_2004, galam_et_al_1982}.
 Models in this field consist of systems with a large number of interacting economic agents, and we refer to~\cite{yakovenko_et_al_2009} for a review.
 The models we consider are simple variants of the model introduced and studied via numerical simulations in~\cite{dragulescu_yakovenko_2000}.
 See also~\cite{bennati_1988, bennati_1993}.
 Their model consists of a system of~$N$ interacting economic agents that are characterized by the number of dollars they possess.
 The system evolves in discrete time as follows:
 at each time step, one agent chosen uniformly at random gives one dollar to another agent again chosen uniformly at random, unless the first agent has no
 money in which case nothing happens. \\
\indent The main idea of econophysics, and more generally sociophysics, is to view human beings as particles, and collisions between two particles as
 interactions between two individuals.
 The fundamental law of equilibrium statistical mechanics is the so-called Boltzmann-Gibbs distribution which states that the probability~$p_e$ that
 a particle has energy~$e$ is well approximated by the exponential random variable.
 More precisely,
 $$ p_e \approx \mu \,e^{- \mu e} \quad \hbox{where} \quad \mu = 1/T = \,\hbox{inverse of the temperature}. $$
 The numerical simulations in~\cite{dragulescu_yakovenko_2000} suggest that this principle extends to the distribution of money in the model above.
 More precisely, letting~$p_d$ be the probability that an agent has~$d$ dollars at equilibrium and letting~$T$ be the average number of dollars per agent,
\begin{equation}
\label{eq:exponential}
  p_d \approx \mu \,e^{- \mu d} \quad \hbox{where} \quad \mu = 1/T.
\end{equation}
 This holds when both the total number of agents and the average number of dollars per agent are large.
 Note that the amount of money an agent has in the context of econophysics can be viewed as the analog of the energy of a particle in physics.
 Also, by analogy with physics, the average number of dollar per agent~$T$ is called the money temperature in econophysics. \\
\indent As far as we know, the convergence to the exponential distribution has only been obtained via numerical simulations.
 This paper gives in contrast a rigorous proof of this result and shows that the convergence is universal in the sense that
 it holds regardless of the underlying network of interactions, i.e., when each agent can only interact with a fixed set of neighbors.
 We also consider a natural variant of this model where agents buy with a probability proportional to the amount of money they have and prove that, in
 this context, the distribution of money converges instead to the~Poisson distribution.
 This again holds for fairly general networks of interactions.


\section{Model description}
\label{sec:models}

\indent To describe our models formally, let~$G = (V, E)$ be a finite connected graph.
\begin{itemize}
 \item Each vertex represents an economic agent and we let~$N = \card (V)$ be the total number of agents present in the system. \vspace*{3pt}
 \item The edge set~$E$ has to be thought of as an interaction network, thus modeling how the agents interact with each other, and
       we call two agents nearest neighbors if the corresponding vertices of the graph are connected by an edge. \vspace*{3pt}
 \item Each agent is characterized by the amount of money she owns, which we assume to be an integer when measured in number of dollars, and we
       let~$M$ be the total number of dollars present in the system at all times (conservative system).
\end{itemize}
 The models we are interested in are discrete-time Markov chains that keep track of the amount of money each of the agents owns.
 Under our assumption that this amount of money is integer-valued, the state at time~$t \in \N$ is a spatial configuration
 $$ \xi_t : V \to \N \quad \hbox{where} \quad \xi_t (x) = \hbox{the number of dollars agent~$x$ has}. $$
 Since the total amount of money in the system is preserved by the dynamics, the state space of the Markov chains consists of the
 following subset of spatial configurations:
\begin{equation}
\label{eq:state-space}
  \begin{array}{l} \A_{N, M} = \{\xi \in \N^V : \sum_{x \in V} \,\xi (x) = M \}. \end{array}
\end{equation}
 In both models, the dynamics consists in moving one dollar from a randomly chosen vertex to a randomly chosen neighbor at each time step.
 In particular, thinking of the graph as a directed graph where each edge~$\{x, y \}$ can have two different orientations~$\overset{\to}{xy}$
 and~$\overset{\to}{yx}$, at each time step, the system jumps from configuration~$\xi$ to one of the configurations
\begin{equation}
\label{eq:transition}
  \xi^{\overset{\to}{xy}} (z) = \left\{\begin{array}{lcl} \xi (z) - \ind \{z = x \} + \ind \{z = y \} & \hbox{when} & \xi (x) \neq 0 \vspace*{2pt} \\
                                                          \xi (z)                                     & \hbox{when} & \xi (x)   =  0 \end{array} \right.
\end{equation}
 for some~$\{x, y \} \in E$.
 The only difference between the two models is that, in the first model, all the pairs of neighbors are equally likely to trade
 at each time step whereas, in the second model, each dollar is equally likely to be spent at each time step, meaning that agents buy with
 a probability proportional to the number of dollars they have. \vspace*{5pt} \\
\noindent {\bf Model 1} -- At each time step, we choose an oriented edge, say~$\overset{\to}{xy}$, uniformly at random and, if there is at least
 one dollar at~$x$, move one dollar from vertex~$x$ to vertex~$y$.
 This is formally described by the discrete-time Markov chain with transition probabilities
 $$ p (\xi, \xi^{\overset{\to}{xy}}) = \frac{1}{2 \card (E)} \quad \hbox{for all} \quad \{x, y \} \in E. $$
 Note that the model introduced in~\cite{dragulescu_yakovenko_2000} is simply the particular case obtained by assuming that the connected graph~$G$
 is the complete graph with~$N$ vertices. \vspace*{5pt} \\
\noindent {\bf Model 2} -- At each time step, we choose one dollar uniformly at random from the system and move it to one of the nearest neighbors
 chosen uniformly at random.
 This is now formally described by the discrete-time Markov chain with transition probabilities
 $$ p (\xi, \xi^{\overset{\to}{xy}}) = \frac{\xi (x)}{M \deg (x)} \quad \hbox{for all} \quad \{x, y \} \in E $$
 where the target configuration is again defined as in~\eqref{eq:transition} and where~$\deg (x)$ refers to the degree of vertex~$x$, i.e., the
 number of neighbors of that vertex.


\section{Main results}
\label{sec:results}

\begin{figure}[h!]
\centering
\includegraphics[width=0.90\textwidth]{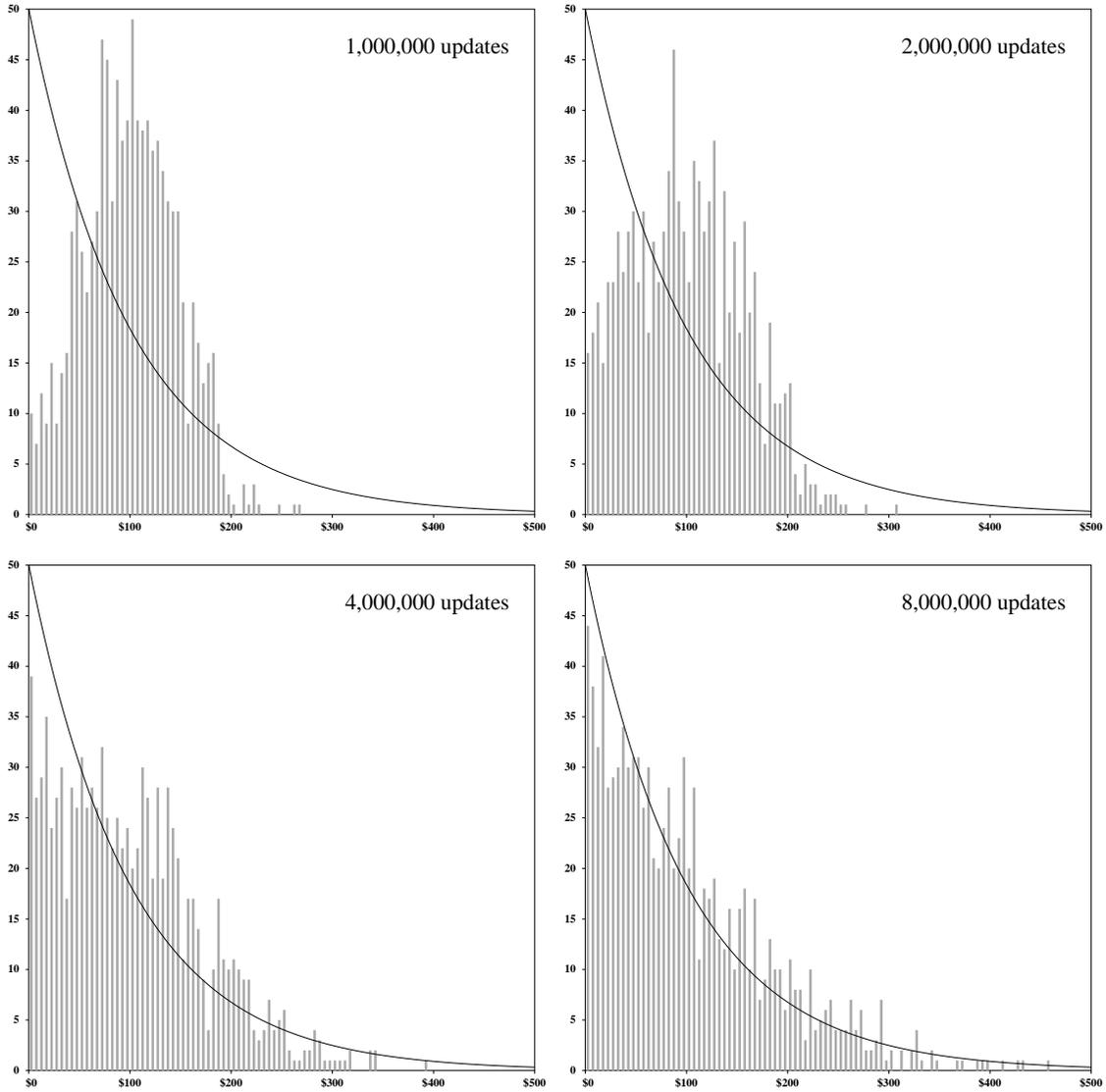}
 \caption{\upshape Simulation results for Model~1 on the complete graph with~1000 vertices, each starting with~$\$100$.
  The gray histograms represent the distribution of money after the number of updates indicated in the top-right corner of the pictures.
  The black solid curve is the limiting exponential distribution found in Theorem~\ref{th:link}.}
\label{fig:exp-dist}
\end{figure}

\begin{figure}[h!]
\centering
\includegraphics[width=0.90\textwidth]{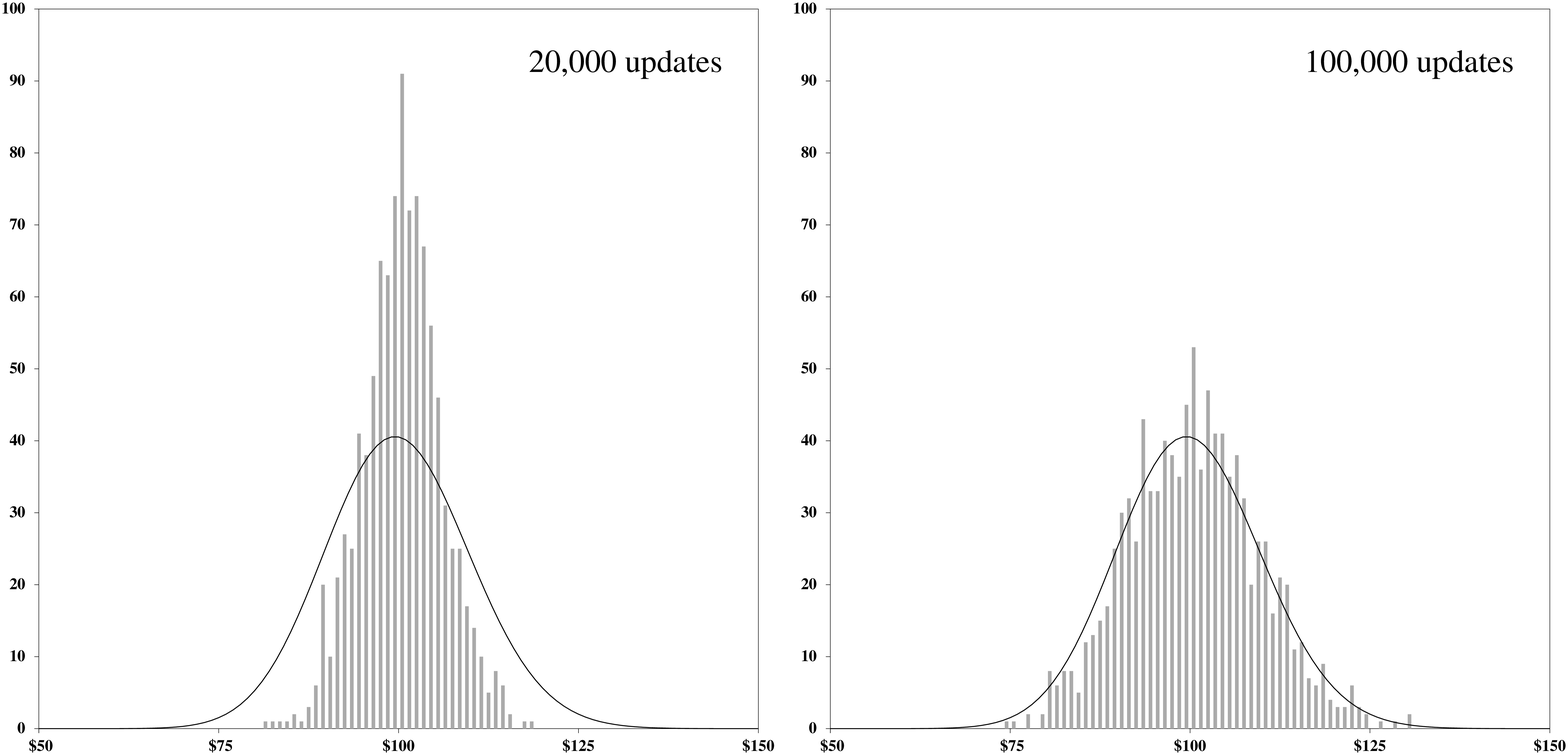}
 \caption{\upshape Simulation results for Model~2 on the complete graph with~1000 vertices, each starting with~$\$100$.
  The gray histograms represent the distribution of money after the number of updates indicated in the top-right corner of the pictures.
  The black solid curve is the limiting Poisson distribution found in Theorem~\ref{th:dollar}.}
\label{fig:poisson-dist}
\end{figure}

\indent For both models, we first study the limiting behavior for all values of the number~$N$ of individuals as time goes to
 infinity and then simplify the probability that an individual has~$d$ dollars at equilibrium in the large population limit.
 Recall also that~$M$ denotes the total amount of money present in the system at all times. \vspace*{5pt} \\
\noindent {\bf Model 1} -- Because the underlying network of interactions is a finite connected graph, the first model is a
 finite irreducible discrete-time Markov chain.
 The process turns out to also be aperiodic from which it follows that there is a unique stationary distribution to which the
 process converges starting from any initial configuration.
 Using time-reversibility and combinatorial techniques, this stationary distribution can be computed explicitly for all~$N$.
 Taking the large population limit as~$N \to \infty$, some basic algebra allows to further simplify the expression of the
 stationary distribution.
 More precisely, we have the following theorem.
\begin{theorem}[exponential distribution] --
\label{th:link}
 For model~1,
 $$ \lim_{t \to \infty} P (\xi_t (x) = d) = {M + N - d - 2 \choose N - 2} \bigg/ {M + N - 1 \choose N - 1}. $$
 In particular, for all fixed~$d$,
 $$ \lim_{N \to \infty} \,\lim_{t \to \infty} P (\xi_t (x) = d) = \frac{e^{- d/T}}{T} + o \bigg(\frac{1}{T} \bigg) \quad \hbox{where} \quad T = M/N. $$
\end{theorem}
 Note that the second part of the theorem implies that, when the money temperature, i.e., the average number of dollars per
 individual, is large, the stationary distribution is well approximated by the exponential distribution~\eqref{eq:exponential},
 which proves the result conjectured in~\cite{dragulescu_yakovenko_2000} and shows that their result extends to spatially explicit
 models where the economic agents interact locally on a general connected graph, rather than globally.
 See Figure~\ref{fig:exp-dist} for simulation results showing the convergence to the exponential distribution when~$G$ is the complete graph. \vspace*{5pt} \\
\noindent {\bf Model 2} -- Like the first model, the second model is a finite irreducible discrete-time Markov chain.
 This implies that there exists a unique stationary distribution and that the fraction of time a given individual has a given
 amount of money converges in the long run to the corresponding coordinate of the stationary distribution.
 The process, however, is not aperiodic in general but aperiodicity holds for instance if the graph has an odd cycle.
 Using again time-reversibility, we prove that this stationary distribution is binomial.
\begin{theorem}[Poisson distribution] --
\label{th:dollar}
 For model~2,
 $$ \lim_{t \to \infty} \,\frac{1}{t} \,\sum_{s = 0}^{t - 1} \,\ind \{\xi_s (x) = d \} =
         {M \choose d} \bigg(\frac{\deg (x)}{\sum_{z \in V} \deg (z)} \bigg)^d \bigg(1 - \frac{\deg (x)}{\sum_{z \in V} \deg (z)} \bigg)^{M - d}. $$
 In particular, on all regular graphs and for all fixed~$d$,
 $$ \lim_{N \to \infty} \,\lim_{t \to \infty} \,\frac{1}{t} \,\sum_{s = 0}^{t - 1} \,\ind \{\xi_s (x) = d \} =
    \frac{T^d}{d!} \,e^{- T} \quad \hbox{where} \quad T = M/N. $$
\end{theorem}
 The second part of the theorem shows that the fraction of time an individual owns~$d$ dollars simplifies and converges to the Poisson
 distribution with mean~$T$, the money temperature, in the large population limit as~$N \to \infty$.
 See Figure~\ref{fig:poisson-dist} for simulation results showing the convergence to the Poisson distribution when the graph~$G$ is the complete graph. \\
\indent The rest of this paper is devoted to proofs, with Section~\ref{sec:exponential} focusing on the first model and Section~\ref{sec:poisson}
 focusing on the second model.


\section{Proof of Theorem~\ref{th:link} (exponential distribution)}
\label{sec:exponential}

\indent To establish Theorem~\ref{th:link}, we first prove that Model~1 has a unique stationary distribution to which the process converges
 starting from any initial configuration and use time reversibility to show that this distribution is the uniform random variable on the
 state space~$\A_{N, M}$.
 The first part of the theorem easily follows by also counting the total number of configurations with~$M$ dollars while the second part can
 be deduced from the first part using some basic algebra. \\
\indent For each~$\xi : V \to \N$ and~$x \in V$, we let
 $$ \xi^x (z) = \xi (z) + \ind \{z = x \} $$
 be the configuration obtained from~$\xi$ by adding one dollar at vertex~$x$.
 Since each configuration with~$M$ dollars can be obtained from a configuration with~$M - 1$ dollars with the addition of one dollar
 at a specific vertex, we have
 $$ \A_{N, M} = \{\xi^x : \xi \in \A_{N, M - 1} \ \hbox{and} \ x \in V \} $$
 which we will use throughout the proofs.
 Note also that
 $$ (\xi^x)^{\overset{\to}{xy}} = \xi^y \in \A_{N, M} \quad \hbox{for all} \quad \xi \in \A_{N, M - 1} \ \hbox{and} \ \{x, y \} \in E. $$
 In the next lemma, we prove that Model~1 has a unique stationary distribution to which the process converges starting from any initial
 configuration.
\begin{lemma} --
\label{lem:convergence-link}
 Model~1 has a unique stationary distribution~$\pi$ and
 $$ \begin{array}{l} \lim_{t \to \infty} P (\xi_t = \xi) = \pi (\xi) \quad \hbox{for all configurations} \quad \xi, \xi_0 \in \A_{N, M}. \end{array} $$
\end{lemma}
\begin{proof}
 To prove the result, the main objective is to show that the discrete-time Markov chain~1 is both irreducible and aperiodic. \vspace*{5pt} \\
 {\bf Irreducibility} --
      Irreducibility follows from the fact that the graph~$G$ is connected.
      Indeed, for each pair of vertices~$(x, y)$, there exist
      $$ x = x_0, x_1, \ldots, x_t = y \quad \hbox{such that} \quad \{x_i, x_{i + 1} \} \in E \ \hbox{for all} \ i = 0, 1, \ldots, t - 1. $$
      In particular, for all~$\xi \in \A_{N, M - 1}$,
      $$ \begin{array}{rcl}
          p_t (\xi^x, \xi^y) & \n = \n & P (\xi_t = \xi^y \,| \,\xi_0 = \xi^x) \vspace*{4pt} \\
                          & \n \geq \n & p (\xi^{x_0}, \xi^{x_1}) \,p (\xi^{x_1}, \xi^{x_2}) \,\cdots \,p (\xi^{x_{n - 1}}, \xi^{x_n}) \vspace*{4pt} \\
                             & \n = \n & p (\xi^{x_0}, (\xi^{x_0})^{\overset{\to}{x_0 x_1}}) \,p (\xi^{x_1}, (\xi^{x_1})^{\overset{\to}{x_1 x_2}}) \,\cdots \,
                                         p (\xi^{x_{n - 1}}, (\xi^{x_{n - 1}})^{\overset{\to}{x_{n - 1} x_n}}) > 0, \end{array} $$
      showing that the two configurations~$\xi^x$ and~$\xi^y$ communicate.
      Using a simple induction, we deduce that, letting~$\xi \in \A_{N, 0}$ be the all-zero configuration and
      $$ (x_1, x_2, x_3, \ldots, x_M), (y_1, y_2, y_3, \ldots, y_M) \in V^M $$
      the two configurations
      $$ (\cdots ((\xi^{x_1})^{x_2})^{x_3} \cdots )^{x_M} \quad \hbox{and} \quad (\cdots ((\xi^{y_1})^{y_2})^{y_3} \cdots )^{y_M} $$
      also communicate.
      Since all the configurations in~$\A_{N, M}$ can be obtained from the all-zero configuration by adding~$M$ dollars, we deduce that all the configurations communicate, which by definition
      means that the process is irreducible. \vspace*{5pt} \\
 {\bf Aperiodicity} --
      For each~$\xi \in \A_{N, M}$ with~$\xi (x) = 0$,
      $$ \xi^{\overset{\to}{xy}} = \xi \quad \hbox{for all} \quad \{x, y \} \in E. $$
      In particular, for~$\xi \in \A_{N, M}$ such that~$\xi (z) = 0$ for some~$z \in V$,
      $$ \begin{array}{l} p (\xi, \xi) = \sum_{z \in V} \deg (z) \,\ind \{\xi (z) = 0 \} / (2 \card (E)) > 0, \end{array} $$
      showing that configurations with at least one vertex with zero dollar have period one.
      Since the process is also irreducible, all the configurations must have the same period one, from which it follows that the process is aperiodic. \vspace*{5pt} \\
 Irreducibility and the fact that the state space~$\A_{N, M}$ is finite imply the existence and uniqueness of a stationary distribution~$\pi$.
 Aperiodicity also implies that, regardless of the initial configuration, the probability that the process is in configuration~$\xi$ converges to~$\pi (\xi)$.
 For a proof of these two classical results, we refer to~\cite[Section~1.7]{durrett_2012}.
\end{proof} \\ \\
 The next lemma shows that the unique stationary distribution~$\pi$ is the uniform distribution on the state space~$\A_{N, M}$ of the process.
\begin{lemma} --
\label{lem:uniform-link}
 We have~$\pi = \uniform (\A_{N, M})$.
\end{lemma}
\begin{proof}
 Depending on whether the oriented edge selected at random starts from a vertex with zero dollar or not, the configuration either remains unchanged or is obtained from the configuration at the
 previous time step by moving one dollar along the oriented edge.
 In equations, this means that if the transition probability~$p (\xi, \xi') > 0$ then we have the following alternative:
\begin{itemize}
 \item $\xi = \xi'$ with~$\xi (z) = 0$ for some~$z \in V$ in which case
       $$ \begin{array}{l} p (\xi, \xi') = p (\xi, \xi) = \sum_{z \in V} \deg (z) \,\ind \{\xi (z) = 0 \} / (2 \card (E)), \end{array} $$
 \item $\xi = \eta^x$ and~$\xi' = \eta^y$ for some~$\eta \in \A_{N, M - 1}$ and~$\{x, y \} \in E$ in which case
       $$ \begin{array}{l} p (\xi, \xi') = p (\eta^x, \eta^y) = 1 / (2 \card (E)). \end{array} $$
\end{itemize}
 This shows in particular that
\begin{equation}
\label{eq:uniform-link-1}
  p (\xi, \xi') \neq 0 \quad \hbox{if and only if} \quad p (\xi', \xi) \neq 0.
\end{equation}
 Also, when~$p (\xi, \xi') \neq 0$ with~$\xi \neq \xi'$ and~$\pi = \uniform (\A_{N, M})$,
 $$ \pi (\xi) \,p (\xi, \xi') = 1 / (2 \card (E) \card (\A_{N, M})) = \pi (\xi') \,p (\xi', \xi). $$
 The left and right-hand sides are trivially equal when~$\xi = \xi'$ while~\eqref{eq:uniform-link-1} shows that the equality also holds when~$p (\xi, \xi') = 0$.
 This implies that the process is time reversible and that the uniform distribution~$\pi$ is indeed a stationary distribution since
 $$ \begin{array}{rcl}
      P_{\pi} (\xi_1 = \xi) & \n = \n & \sum_{\xi' \in \A_{N, M}} \pi (\xi') \,p (\xi', \xi) = \sum_{\xi' \in \A_{N, M}} \pi (\xi) \,p (\xi, \xi') \vspace*{4pt} \\
                            & \n = \n & \pi (\xi) \,\sum_{\xi' \in \A_{N, M}} p (\xi, \xi') = \pi (\xi). \end{array} $$
 This completes the proof.
\end{proof} \\ \\
 It follows from the previous lemma that
 $$ \pi (\xi) = \frac{1}{\card (\A_{N, M})} \quad \hbox{for all} \quad \xi \in \A_{N, M}. $$
 In particular, to obtain a more explicit expression of the stationary distribution, it suffices to compute the number of configurations.
 This is done in the next lemma.
\begin{lemma} --
\label{lem:number-link}
 For all positive integers~$N, M \in \N^*$,
 $$ \card (\A_{N, M}) = {M + N - 1 \choose N - 1}. $$
\end{lemma}
\begin{proof}
 Write~$V = \{x_1, x_2, \ldots, x_N \}$ and, for each~$\xi \in \A_{N, M}$, let
 $$ \phi (\xi) = \{\xi (x_1) + 1, \xi (x_1) + \xi (x_2) + 2, \ldots, \xi (x_1) + \cdots + \xi (x_{N - 1}) + N - 1 \}. $$
 This defines a function~$\phi : \A_{N, M} \to \B_{N, M}$ where
 $$ \B_{N, M} = \hbox{set of subsets of~$\{1, 2, \ldots, M + N - 1 \}$ with~$N - 1$ elements} $$
 and we now prove that this function is bijective. \vspace*{5pt} \\
 {\bf Injectivity} -- Let~$\xi, \xi' \in \A_{N, M}$ with~$\phi (\xi) = \phi (\xi')$. Then,
      $$ \xi (x_i) = \xi' (x_i) \quad \hbox{for all} \quad i = 1, 2, \ldots, N - 1. $$
      Since both configurations contain~$M$ dollars, we also have
      $$ \xi (x_N) = M - \xi (x_1) - \cdots - \xi (x_{N - 1}) = M - \xi' (x_1) - \cdots - \xi' (x_{N - 1}) = \xi' (x_N) $$
      showing that~$\xi = \xi'$ and that~$\phi$ is injective. \vspace*{5pt} \\
 {\bf Surjectivity} -- Let~$B \in \B_{N, M}$ and write
      $$ B = \{n_1, n_2, \ldots, n_{N - 1} \} \quad \hbox{with} \quad 1 \leq n_1 < n_2 < \cdots < n_{N - 1} \leq M + N - 1. $$
      Then, define the configuration~$\xi : V \to \N$ as
      $$ \xi (x_i) = \left\{\begin{array}{lcl} n_1 - 1               & \hbox{for} & i = 1 \vspace*{2pt} \\
                                               n_i - n_{i - 1} - 1   & \hbox{for} & i = 2, 3, \ldots, N - 1  \vspace*{2pt} \\
                                               M + N - n_{N - 1} - 1 & \hbox{for} & i = N. \end{array} \right. $$
      We easily check that~$\xi \in \A_{N, M}$ and~$\phi (\xi) = B$, which shows surjectivity. \vspace*{5pt} \\
 In conclusion, we have
 $$ \card (\A_{N, M}) = \card (\B_{N, M}) = {M + N - 1 \choose N - 1} $$
 where the first equation follows from the bijectivity of~$\phi$, while the second equation is obvious in view of the definition of the set~$\B_{N, M}$.
\end{proof} \\ \\
 Using Lemmas~\ref{lem:convergence-link}--\ref{lem:number-link}, we can now prove Theorem~\ref{th:link}. \\ \\
\begin{proofof}{Theorem~\ref{th:link}}
 It follows from Lemmas~\ref{lem:convergence-link}--\ref{lem:uniform-link} that
\begin{equation}
\label{eq:link-1}
  \begin{array}{rcl}
  \lim_{t \to \infty} P (\xi_t (x) = d) & \n = \n & \pi (\{\xi \in \A_{N, M} : \xi (x) = d \}) \vspace{4pt} \\
                                        & \n = \n & \card \,\{\xi \in \A_{N, M} : \xi (x) = d \} / \card (\A_{N, M}). \end{array}
\end{equation}
 for all~$(x, d) \in V \times \{0, 1, \ldots, M \}$, regardless of the initial configuration.
 In addition, the number of configurations with exactly~$d$ dollars at vertex~$x$ is given by
\begin{equation}
\label{eq:link-2}
  \card \,\{\xi \in \A_{N, M} : \xi (x) = d \} = \card (\A_{N - 1, M - d}).
\end{equation}
 Combining~\eqref{eq:link-1}--\eqref{eq:link-2} and using Lemma~\ref{lem:number-link}, we get
\begin{equation}
\label{eq:link-3}
  \lim_{t \to \infty} P (\xi_t (x) = d) = \displaystyle \frac{\card (\A_{N - 1, M - d})}{\card (\A_{N, M})} = {M + N - d - 2 \choose N - 2} \bigg/ {M + N - 1 \choose N - 1}
\end{equation}
 which proves the first part of the theorem.
 To deduce the second part of the theorem, we first rewrite the right-hand side of~\eqref{eq:link-3} as
 $$ \begin{array}{l}
    \displaystyle \frac{(M + N - d - 2)!}{(N - 2)! (M - d)!} \ \frac{(N - 1)! M!}{(M + N - 1)!} \vspace*{8pt} \\ \hspace*{20pt} = \
    \displaystyle \frac{(N - 1)!}{(N - 2)!} \ \frac{M!}{(M - d)!} \ \frac{(M + N - d - 2)!}{(M + N - 1)!} \vspace*{8pt} \\ \hspace*{20pt} = \
    \displaystyle \frac{M \,(M - 1) \cdots (M - d + 1)(N - 1)}{(M + N - 1)(M + N - 2) \cdots (M + N - d - 1)}. \end{array} $$
 Letting~$T = M/N$ be the average number of dollars per vertex, which by analogy with classical physics is called the money temperature, and observing that both the numerator and the denominator
 are the product of~$d + 1$ terms, for~$d \in \N$ fixed,
 $$ \begin{array}{l}
    \displaystyle \lim_{N \to \infty} \,\lim_{t \to \infty} \,P (\xi_t (x) = d) = \frac{N M^d}{(M + N)^{d + 1}} \vspace*{8pt} \\ \hspace*{40pt} = \
    \displaystyle \bigg(\frac{1}{T + 1} \bigg) \bigg(\frac{T}{T + 1} \bigg)^d = \bigg(\frac{1}{T + 1} \bigg) \,e^{- d \ln \left(1 + \frac{1}{T} \right)}. \end{array} $$
 In particular, for large money temperatures,
 $$ \lim_{N \to \infty} \,\lim_{t \to \infty} \,P (\xi_t (x) = d) =
    \bigg(\frac{1}{T} + o \bigg(\frac{1}{T} \bigg) \bigg) \,e^{- d \left(\frac{1}{T} + o \left(\frac{1}{T} \right) \right)} =
    \frac{e^{- d/T}}{T} + o \bigg(\frac{1}{T} \bigg) $$
 showing that, at least when the temperature is high and in the large population limit, the number of dollars at a given vertex at equilibrium is well
 approximated by the exponential random variable with parameter~$1/T$.
 This completes the proof.
\end{proofof}


\section{Proof of Theorem~\ref{th:dollar} (Poisson distribution)}
\label{sec:poisson}

\indent To establish Theorem~\ref{th:dollar}, which focuses on the second model, we again start by proving the existence and uniqueness
 of the stationary distribution.
 The process, however, might not be aperiodic, so we only have convergence of the fraction of time spent in each state rather than convergence
 of the multi-step transition probabilities.
 For model~2, the stationary distribution is the multinomial random variable, which can be guessed from the stationary distribution
 of the symmetric random walk on the connected graph~$G$.
 Both parts of the theorem easily follow.
\begin{lemma} --
\label{lem:convergence-dollar}
 Model~2 has a unique stationary distribution~$\pi$ and
\begin{equation}
\label{eq:convergence-dollar-0}
 \lim_{t \to \infty} \,\frac{1}{t} \,\sum_{s = 0}^{t - 1} \,\ind \{\xi_s (x) = d \} = \n \sum_{\xi \in \A_{N, M}} \n \pi (\xi) \,\ind \{\xi (x) = d \} \quad \hbox{for all} \quad \xi_0 \in \A_{N, M}.
\end{equation}
\end{lemma}
\begin{proof}
 Using again that the graph~$G$ is connected and following the same argument as in the proof of Lemma~\ref{lem:convergence-link}, we prove that the
 process is irreducible.
 Since in addition the state space is finite, there exists a unique stationary distribution~$\pi$ and, by \cite[Theorem~1.23]{durrett_2012},
\begin{equation}
\label{eq:convergence-dollar-1}
 \lim_{t \to \infty} \,\frac{1}{t} \,\sum_{s = 0}^{t - 1} \,f (\xi_s) = \n \sum_{\xi \in \A_{N, M}} \n f (\xi) \,\pi (\xi) \quad \hbox{for all} \quad \xi_0 \in \A_{N, M}
\end{equation}
 and all bounded functions~$f : \A_{N, M} \to \R$. Taking
 $$ f (\xi) = \ind \{\xi (x) = d \} \quad \hbox{for a fixed pair} \ (x, d) \in V \times \{0, 1, \ldots, M \}, $$
 equation~\eqref{eq:convergence-dollar-1} becomes~\eqref{eq:convergence-dollar-0}.
 This completes the proof.
\end{proof} \\ \\
 Note that we have the stronger convergence
 $$ \lim_{t \to \infty} \,P (\xi_s (x) = d) = \n \sum_{\xi \in \A_{N, M}} \n \pi (\xi) \,\ind \{\xi (x) = d \} \quad \hbox{for all} \quad \xi_0 \in \A_{N, M} $$
 whenever the process is also aperiodic.
 This is not true for all connected graphs~$G$ but aperiodicity holds for instance if the graph has an odd cycle.
 The next step is to find an explicit expression of the distribution~$\pi$.
 Thinking of the money circulating in the system as a set of~$M$ one-dollar bills, the process~$(X_t)$ that keeps track of the location of a given bill
\begin{itemize}
 \item stays put at each time step with probability~$1 - 1/M$ and \vspace*{3pt}
 \item jumps according to the symmetric random walk on~$G$ with probability~$1/M$.
\end{itemize}
 This process is known to be reversible with stationary distribution
 $$ \bar \pi (w) = \frac{\deg (w)}{\sum_{z \in V} \deg (z)} \quad \hbox{for all} \quad w \in V. $$
 In particular, a good candidate for the stationary distribution~$\pi$ is the distribution in which each bill is independently at vertex~$w$ with
 probability~$\bar \pi (w)$, i.e.,
\begin{equation}
\label{eq:multinomial-dollar-0}
 \pi (\xi) = {M \choose \xi (1), \ldots, \xi (N)} \,\prod_{w \in V} \,(\bar \pi (w))^{\xi (w)} \quad \hbox{where} \quad \bar \pi (w) = \frac{\deg (w)}{\sum_{z \in V} \deg (z)}
\end{equation}
 for all~$\xi \in \A_{N, M}$.
 This is proved in the next lemma.
\begin{lemma} --
\label{lem:multinomial-dollar}
 The distribution~$\pi$ given in~\eqref{eq:multinomial-dollar-0} is stationary for model~2.
\end{lemma}
\begin{proof}
 First, we observe that, for all~$\eta \in \A_{N, M - 1}$ and~$\{x, y \} \in E$,
\begin{equation}
\label{eq:multinomial-dollar-1}
  p (\eta^x, \eta^y) = P (\xi_{t + 1} = \eta^y \,| \,\xi_t = \eta^x) = \frac{\eta^x (x)}{M \deg (x)} = \frac{\eta (x) + 1}{M \deg (x)}.
\end{equation}
 In addition, for all~$\eta \in \A_{N, M - 1}$ and~$x \in V$,
\begin{equation}
\label{eq:multinomial-dollar-2}
 \begin{array}{rcl}
  \pi (\eta^x) & \n = \n &
  \displaystyle {M \choose \eta^x (1), \ldots, \eta^x (N)} \,\prod_{w \in V} \,(\bar \pi (w))^{\eta^x (w)} \vspace*{8pt} \\ & \n = \n &
  \displaystyle {M \choose \eta (1), \ldots, \eta (N)} \bigg(\frac{1}{\eta (x) + 1} \bigg) \bigg(\prod_{w \in V} \,(\bar \pi (w))^{\eta (w)} \bigg) \,\bar \pi (x).
 \end{array}
\end{equation}
 Combining~\eqref{eq:multinomial-dollar-1}--\eqref{eq:multinomial-dollar-2}, we deduce that, for all~$\{x, y \} \in E$,
 $$ \frac{\pi (\eta^x)}{\pi (\eta^y)} = \frac{\eta (y) + 1}{\eta (x) + 1} \ \frac{\bar \pi (x)}{\bar \pi (y)}
                                      = \frac{\eta (y) + 1}{\eta (x) + 1} \ \frac{\deg (x)}{\deg (y)} = \frac{p (\eta^y, \eta^x)}{p (\eta^x, \eta^y)} $$
 while it is trivial that
 $$ p (\xi, \xi') = 0 \quad \hbox{for all} \quad (\xi, \xi') \notin \{(\eta^x, \eta^y) : \eta \in \A_{N, M - 1} \ \hbox{and} \ \{x, y \} \in E \}. $$
 This shows that the process is time reversible and, as in the proof of Lemma~\ref{lem:uniform-link}, that the distribution given in~\eqref{eq:multinomial-dollar-0}
 is indeed the stationary distribution of model~2.
\end{proof} \\ \\
 Using Lemmas~\ref{lem:convergence-dollar}--\ref{lem:multinomial-dollar}, we can now prove Theorem~\ref{th:dollar}. \\ \\
\begin{proofof}{Theorem~\ref{th:dollar}}
 Fix a vertex~$x \in V$ and write
 $$ V = \{x, w_1, w_2, \ldots, w_{N - 1} \}. $$
 By Lemma~\ref{lem:multinomial-dollar}, for all~$\xi \in \A_{N, M}$ such that~$\xi (x) = d$,
\begin{equation}
\label{eq:2}
  \begin{array}{rcl}
  \pi (\xi) & \n = \n & \displaystyle {M \choose d, \xi (w_1), \ldots, \xi (w_{N - 1})} \bigg(\prod_{w \neq x} \,(\bar \pi (w))^{\xi (w)} \bigg) (\bar \pi (x))^d \vspace*{8pt} \\
            & \n = \n & \displaystyle {M \choose d} {M - d \choose \xi (w_1), \ldots, \xi (w_{N - 1})} \bigg(\prod_{i = 1}^{N - 1} \,(\bar \pi (w_i))^{\xi (w_i)} \bigg) (\bar \pi (x))^d. \end{array}
\end{equation}
 By~\eqref{eq:2} and the multinomial theorem, the right-hand side of~\eqref{eq:convergence-dollar-0} becomes
 $$ \begin{array}{l}
    \displaystyle \sum_{\xi \in \A_{N, M}} \n {M \choose d} {M - d \choose \xi (w_1), \ldots, \xi (w_{N - 1})} \bigg(\prod_{i = 1}^{N - 1} \,(\bar \pi (w_i))^{\xi (w_i)} \bigg) (\bar \pi (x))^d \,\ind \{\xi (x) = d \} \vspace*{8pt} \\ \hspace*{20pt} =
    \displaystyle {M \choose d} (\bar \pi (x))^d \n \sum_{\xi (w_1) + \cdots + \xi (w_{N - 1}) = M - d} {M - d \choose \xi (w_1), \ldots, \xi (w_{N - 1})} \bigg(\prod_{i = 1}^{N - 1} \,(\bar \pi (w_i))^{\xi (w_i)} \bigg) \vspace*{8pt} \\ \hspace*{20pt} =
    \displaystyle {M \choose d} (\bar \pi (x))^d \,\bigg(\sum_{i = 1}^{N - 1} \,\bar \pi (w_i) \bigg)^{M - d} =
    \displaystyle {M \choose d} (\bar \pi (x))^d \,(1 - \bar \pi (x))^{M - d}. \end{array} $$
 Applying Lemma~\ref{lem:convergence-dollar} and recalling~\eqref{eq:multinomial-dollar-0}, we deduce that
 $$ \begin{array}{l}
    \displaystyle \lim_{t \to \infty} \,\frac{1}{t} \,\sum_{s = 0}^{t - 1} \,\ind \{\xi_s (x) = d \} =
    \displaystyle {M \choose d} (\bar \pi (x))^d \,(1 - \bar \pi (x))^{M - d} \vspace*{8pt} \\ \hspace*{80pt} =
    \displaystyle {M \choose d} \bigg(\frac{\deg (x)}{\sum_{z \in V} \deg (z)} \bigg)^d \bigg(1 - \frac{\deg (x)}{\sum_{z \in V} \deg (z)} \bigg)^{M - d} \end{array} $$
 for all~$\xi_0 \in \A_{N, M}$.
 This proves the first part of the theorem.
 To deduce the second part, we simply observe that, for all regular graphs with~$N$ vertices,
 $$ \bar \pi (w) = \frac{\deg (w)}{\sum_{z \in V} \deg (z)} = \frac{1}{N} \quad \hbox{for all} \quad w \in V. $$
 In particular, taking the limit as~$N \to \infty$ and recalling~$T = M/N$, we get
 $$ \begin{array}{l}
    \displaystyle \lim_{N \to \infty} \lim_{t \to \infty} \,\frac{1}{t} \,\sum_{s = 0}^{t - 1} \,\ind \{\xi_s (x) = d \} =
    \lim_{N \to \infty} {NT \choose d} \bigg(\frac{1}{N} \bigg)^d \bigg(1 - \frac{1}{N} \bigg)^{NT - d} \vspace*{8pt} \\ \hspace*{25pt} = \
    \displaystyle \lim_{N \to \infty} \bigg(\frac{NT (NT - 1) \cdots (NT - d + 1)}{d! \,N^d} \bigg) \bigg(1 - \frac{1}{N} \bigg)^{NT}
      = \frac{T^d}{d!} \,e^{- T}. \end{array} $$
 This completes the proof.
\end{proofof}




\end{document}